\documentclass[11pt, letterpaper,final]{amsart}
\usepackage{amsfonts,amsmath, amsthm, amssymb, latexsym, epsfig}
\usepackage[english]{babel}
\usepackage[utf8]{inputenc}
\usepackage[all]{xy}
\usepackage{xspace}
\usepackage{comment}
\usepackage{setspace}
\usepackage{enumerate}
\usepackage{stmaryrd}
\usepackage[mathscr]{eucal}
\usepackage[notcite,notref]{showkeys}

	\advance \topmargin by -\headheight
	\advance \topmargin by -\headsep

	\evensidemargin \oddsidemargin
	\newcommand{\multialg}[1]{\mathcal{M}(#1)\xspace}

	\theoremstyle{plain}
	\newtheorem{thm}{Theorem}[section]
	\newtheorem{lemma}[thm]{Lemma}
	\newtheorem{theorem}[thm]{Theorem}
	\newtheorem{proposition}[thm]{Proposition}
	\newtheorem{corollary}[thm]{Corollary}

	\theoremstyle{definition}
	\newtheorem{definition}[thm]{Definition}
	
	\newtheorem{remark}[thm]{Remark}
	
	\newtheorem{example}[thm]{Example}

	\numberwithin{equation}{section}
	\numberwithin{figure}{section}

\begin{document}
	\title{Quasidiagonal traces on exact $C^\ast$-algebras}
	\author{James Gabe}
        \address{Department of Mathematics and Computer Science \\
        The University of Southern Denmark \\
        Campusvej 55 \\
        DK–5230 Odense M, Denmark}
        \email{jamiegabe123@hotmail.com}

\begin{abstract}
Recently, it was proved by Tikuisis, White and Winter that any faithful trace on a separable, nuclear $C^\ast$-algebras in the UCT class is quasidiagonal.
Building on their work, we generalise the result, and show that any faithful, amenable trace on a separable, exact $C^\ast$-algebras in the UCT class is quasidiagonal.
We also prove that any amenable trace on a separable, exact, quasidiagonal $C^\ast$-algebra in the UCT class is quasidiagonal.
\end{abstract}

\maketitle

\section{Introduction}

The classification of separable, nuclear $C^\ast$-algebras using $K$-theory and traces, the Elliott classification programme, has seen prominent progress over last 25 years.
The most recent success is due to Tikuisis, White and Winter in \cite{TikuisisWhiteWinter-QDnuc} which shows that separable, unital, simple, non-elementary $C^\ast$-algebras in the UCT class with finite nuclear dimension are classified by their Elliott invariant. This very complete classification result builds on a long line of results over the last 25 years, such as \cite{Kirchberg-simple}, \cite{Phillips-classification}, \cite{Lin-classtracialtoprankzero}, \cite{GongLinNiu-classZ-stable}, \cite{ElliottGongLinNiu-classfindec}.

One of the early major successes in the classification programme was the Kirchberg--Phillips theorem \cite{Kirchberg-simple}, \cite{Phillips-classification}, the classification of separable, nuclear, unital, simple, purely infinite $C^\ast$-algebras in the UCT class, by the $K$-groups and the position of the unit in $K_0$.
The UCT class refers to the class of separable $C^\ast$-algebras satisfying the universal coefficient theorem of Rosenberg and Schochet \cite{RosenbergSchochet-UCT}, or equivalently, the class of separable $C^\ast$-algebras which are $KK$-equivalent to abelian $C^\ast$-algebras.

In Kirchberg's approach to the classification result (see \cite{Kirchberg-simple} or \cite{Rordam-book-classification}), much more is actually proved. It follows from this approach, that if $A$ is a separable, exact, unital $C^\ast$-algebra in the UCT class and $B$ is a unital, purely infinite $C^\ast$-algebra, then any pointed homomorphism $\phi \colon (K_\ast(A),[1_A]_0) \to (K_\ast(B),[1_B]_0)$ lifts to a full, unital, \emph{nuclear} $\ast$-homomorphism $A \to B$. Using $KK$-theory (resp.~total $K$-theory) one even obtains uniqueness results for such nuclear $\ast$-homomorphisms up to asymptotic (resp.~approximate) unitary equivalence.

The same idea was employed by Dadarlat in \cite{Dadarlat-morphismstracialAF}. He showed that if $A$ and $B$ are separable, simple, unital, tracially AF $C^\ast$-algebras, and $A$ is exact and in the UCT class, then any pointed, order preserving homomorphism $\phi\colon (K_\ast^+(A),[1_A]_0) \to (K_\ast^+(B),[1_B]_0)$ lifts to a nuclear $\ast$-homomorphism $A\to B$. He also obtains uniqueness of such $\ast$-homomorphisms up to approximate unitary equivalence using total $K$-theory. This reproves the classification of separable, nuclear, unital, simple tracially AF $C^\ast$-algebras in the UCT class by Lin \cite{Lin-classtracialtoprankzero}.

It turns out that this phenomenon often occurs. Rather than considering nuclear $C^\ast$-algebras, one might as well consider nuclear maps for which the domain is an exact $C^\ast$-algebra.
This actually also explains why classification only holds for nuclear $C^\ast$-algebras: by using the methods for nuclear $\ast$-homomorphisms with exact domains, the isomorphism one would construct by the classification results would have to be nuclear. Thus the $C^\ast$-algebras we classify would have to be nuclear by Kirchberg's characterisation of nuclear $C^\ast$-algebras as the $C^\ast$-algebras for which the identity map is nuclear \cite{Kirchberg-CPAP}.

Popa's work on simple quasidiagonal $C^\ast$-algebras in \cite{Popa-findimapprox} was the main inspiration for Lin's definition of tracially AF $C^\ast$-algebras, which is a key component in the classification programme. Quasidiagonality of traces was later introduced by Brown in \cite{Brown-invariantmeans} and has also proved to be play an important part in the classification programme. Very recently, Tikuisis, White and Winter showed in \cite{TikuisisWhiteWinter-QDnuc} that any faithful trace on a separable, nuclear $C^\ast$-algebra in the UCT class is quasidiagonal. This result has several remarkable consequences. It follows that $(1)$: the Blackadar--Kirchberg problem \cite{BlackadarKirchberg-genindlimfindim} has an affirmative answer for simple $C^\ast$-algebras in the UCT class, $(2)$: the Rosenberg conjecture is true (see appendix of \cite{Hadwin-stronglyqd}), and as mentioned above $(3)$: separable, unital, simple, non-elementary $C^\ast$-algebras in the UCT class with finite nuclear dimension are classified by their Elliott invariant.

We prove that the following generalisation of the result by Tikuisis, White and Winter for exact $C^\ast$-algebras is true: any faithful, amenable trace on a separable, exact $C^\ast$-algebra in the UCT class is quasidiagonal. In particular, it follows that a separable, exact, unital $C^\ast$-algebra in the UCT class is quasidiagonal if it has a faithful, amenable trace. This is proved by using the results and methods of \cite{TikuisisWhiteWinter-QDnuc} together with results on nuclear order zero maps. We also show that any (not necessarily faithful) amenable trace on a separable, exact, quasidiagonal $C^\ast$-algebra in the UCT class is quasidiagonal. This partially answers a question of N.~Brown \cite[Question 6.7(2)]{Brown-invariantmeans}, asking if amenable traces are always quasidiagonal.

We raise the following questions motivated by the purely infinite case and the tracially AF case: to which extent can the classification result $(3)$ above be obtained from existence and uniqueness results regarding nuclear maps going out of exact $C^\ast$-algebras? Also, should we consider the full Elliott invariant, or only consider amenable traces in the invariant? 

\subsection*{Acknowledgements} The author is grateful to Aaron Tikuisis, Stuart White and Wilhelm Winter for sharing early drafts of their paper \cite{TikuisisWhiteWinter-QDnuc}, for including Lemma 5.7 in their paper, and for several helpful comments and suggestions for this paper.

\section{Nuclear maps}

\subsection{The basics} In this subsection we mention some standard results about nuclear maps which will be used later. We use the standard abbreviation \emph{c.p.~map} for completely positive map.

\begin{definition}
A c.p.~map $\phi \colon A \to B$ is called \emph{nuclear} if it is a point-norm limit of c.p.~maps factoring through matrix algebras.
\end{definition}

Clearly the set of nuclear maps is point-norm closed, and if you compose a nuclear map with any c.p.~map, then you obtain a nuclear map.

Note that we do not assume that the c.p.~maps factoring through matrix algebras in the above definition are contractive (or unital), as done in several places in the literature, e.g.~in \cite{BrownOzawa-book-approx}. However, as shown in \cite[Lemma 2.3]{GabeRuiz-absrep}, this definition agrees with ours for contractive maps. 

\begin{remark}\label{r:nucopconcon}
If $\phi\colon A \to B$ is a nuclear map, $a_1,\dots,a_n\in A$ and $b_1,\dots,b_n \in B$ then the map
\[
A \ni x \mapsto \sum_{i,j=1}^n b_i^\ast \phi(a_i^\ast x a_j) b_j \in B
\]
is nuclear. To see this, let $r \in M_{1,n}(A)$ be the row vector $r= (a_1 \; \dots \; a_n)$, $c\in M_{n,1}(B)$ be the column vector $(b_1 \; \dots \; b_n)^t$, and $\phi^{(n)} \colon M_n(A) \to M_n(B)$ be the induced c.p.~map. Clearly $\phi^{(n)}$ is nuclear, and thus
\[
x \mapsto \sum_{i,j=1}^n b_i^\ast \phi(a_i^\ast x a_j) b_j = c^\ast \phi^{(n)} (r^\ast x r) c
\] 
is nuclear.
\end{remark}

For a $C^\ast$-algebra $B$ we let $\multialg{B}$ denote the multiplier algebra of $B$. Recall, that $\multialg{B}$ embeds canonically into the bidual $B^{\ast \ast}$.

\begin{definition}
Let $A$ and $B$ be $C^\ast$-algebras, and let $M$ be a von Neumann algebra.
\begin{itemize}
\item A c.p.~map $\phi \colon A \to \multialg{B}$ is called \emph{weakly nuclear} if $b^\ast \phi(-) b \colon A \to B$ is nuclear for every $b\in B$.
\item A c.p.~map $\phi \colon A \to M$ is called \emph{weakly nuclear} if it is a point-weak$\ast$ limit of c.p.~maps factoring through matrix algebras.
\end{itemize}
\end{definition}

It could be confusing that we have two different notions called ``weakly nuclear'', but condition $(2)$ in the following proposition motivates this.
The result is well-known, but we will fill in the proof for completion. 

\begin{proposition}\label{p:weaklynucmaps}
Let $A$ and $B$ be $C^\ast$-algebras.
\begin{itemize}
\item[$(1)$] A c.p.~map $\phi \colon A \to B$ is nuclear if and only if the map $\phi \colon A \to B \subseteq B^{\ast \ast}$ is weakly nuclear.
\item[$(2)$] A c.p.~map $\phi \colon A \to \multialg{B}$ is weakly nuclear (in the multiplier algebra sense) if and only if $\phi \colon A \to \multialg{B} \subseteq B^{\ast \ast}$ is weakly nuclear (in the von Neumann algebra sense).
\end{itemize}
\end{proposition}
\begin{proof}
$(1)$: one implication is trivial. We will only sketch the proof of the other implication, since this is essentially the same proof as \cite[Proposition 2.3.8]{BrownOzawa-book-approx}, that if $A^{\ast\ast}$ is semidiscrete then $A$ is nuclear. For any c.p.~map $\psi\colon M_N \to B^{\ast \ast}$, any normal functionals $\eta_1,\dots, \eta_n$ on $B^{\ast \ast}$ and any $\epsilon >0$, there is c.p.~map $\tilde \psi \colon M_N \to B$ such that
\[
| \rho_k(\psi(x) - \tilde \psi(x)) |< \epsilon, \quad x\in (M_N)_1, k=1,\dots,n.
\]
It follows that $\phi$ is the point-weak$\ast$ limit of c.p.~maps $A \to B$ factoring through matrix algebras. By a standard Hahn--Banach separation argument (see e.g.~\cite[Lemma 2.6]{gabe-cplifting}) it follows that $\phi$ is nuclear.

$(2)$: If $\phi \colon A \to \multialg{B}$ is weakly nuclear and $(b_\alpha)$ is an approximate identity in $B$ of positive contractions, then $(b_\alpha \phi(-) b_\alpha)$ is (by assumption) a net of nuclear maps, which converges point-weak$\ast$ to $\phi \colon A \to \multialg{B} \subseteq B^{\ast \ast}$. Conversely, if $\phi \colon A \to \multialg{B} \subseteq B^{\ast\ast}$ is weakly nuclear (in the von Neumann algebra sense), and if $b\in B$, then $b^\ast \phi (-) b \colon A \to B$ is weakly nuclear as a map into $B^{\ast \ast}$. By $(1)$ it follows that $b^\ast \phi(-) b$ is nuclear, so $\phi$ is weakly nuclear in the multiplier algebra sense.
\end{proof}

\subsection{Nuclear order zero maps}

Recall, that a c.p.~map $\phi \colon A \to B$ is called \emph{order zero} if it preserves orthogonality, i.e.~whenever $a,b\in A$ are positive and $ab=0$, then $\phi(a) \phi(b) =0$.

Winter and Zacharias proved in \cite{WinterZacharias-orderzero} that order zero maps have a certain dilation property. This was used to prove (amongst other things) that there is a one-to-one correspondence between contractive order zero maps $A \to B$ and $\ast$-homomorphisms $C_0(0,1] \otimes A \to B$. We will prove nuclear analogues of these results.

The dilation result \cite[Theorem 2.3]{WinterZacharias-orderzero} of Winter and Zacharias is as follows: For $\phi \colon A \to B$ an order zero map, let $C = C^\ast(\phi(A))$. Then there is a positive element $h \in \multialg{C} \cap C'$ and a $\ast$-homomorphism $\pi_\phi \colon A \to \multialg{C} \cap \{h\}'$ such that $h \pi_\phi(-) = \phi(-)$. A priori $h \pi_\phi(-)$ is a map into $\multialg{C}$, but it factors through $C$.

\begin{remark}\label{r:dilationher}
If we let $D$ be the hereditary $C^\ast$-subalgebra of $B$ generated by $C$, i.e.~$D = \overline{CBC}$, then $\multialg{C} \subseteq \multialg{D}$ canonically. Thus we could assume in the above result, that $\pi_\phi \colon A \to \multialg{D}$ and $h\in \multialg{D}$ satisfy $h \pi_\phi(-) = \phi(-)$. This is convenient when working with nuclear order zero maps.
\end{remark}

\begin{proposition}\label{p:orderzerodilation}
Let $A$ and $B$ be $C^\ast$-algebras, $\phi \colon A \to B$ be a nuclear order zero map, let $C = C^\ast(\phi(A))$ and let $D = \overline{CBC}$. Suppose that $h\in \multialg{C} \cap C' \subseteq \multialg{D}$ is a positive element and $\pi_\phi \colon A \to \multialg{C} \cap \{h\}' \subseteq \multialg{D}$ is a $\ast$-homomorphism, such that $h \pi_\phi(-) = \phi(-)$.
Then $\pi_\phi \colon A \to \multialg{D}$ is weakly nuclear.

In particular, $\pi_\phi\colon A \to \multialg{D} \subseteq D^{\ast \ast}$ is weakly nuclear (in the von Neumann algebra sense).
\end{proposition}
\begin{proof}
Note that $\phi = \pi_\phi(-) h = h \pi_\phi(-)$. Thus 
\[
\phi(a) \pi_\phi(x) \phi(b) = h \pi_\phi(a) \pi_\phi(x) \pi_\phi(b) h = h^{1/2} \pi_\phi(axb) h^{3/2} = h^{1/2} \phi(axb) h^{1/2}
\]
for $a,b,x\in A$.
Our goal is to show that the c.p.~maps $d^\ast \pi_\phi(-) d \colon A \to D$ are nuclear for all $d\in D$.
Since the set of nuclear maps is point-norm closed, it suffices to check this for elements $d$ of the form $d = \sum_{i=1}^n\phi(a_i) b_i$, with $a_i \in A$ and $b_i \in D$, since elements of this form are dense in $D$. Note that $h \in \multialg{C} \subseteq \multialg{D}$ and thus $h^{1/2} b_i \in D$. We get that
\begin{eqnarray*}
A \ni x \mapsto d^\ast\pi_\phi(x) d &=& \sum_{i,j=1}^n b_i^\ast (\phi(a_i)^\ast \pi_\phi(x) \phi(a_j)) b_j \\
&=& \sum_{i,j=1}^n b_i^\ast h^{1/2} \phi(a_i^\ast x a_j) h^{1/2} b_j \\
&=& \sum_{i,j=1}^n (h^{1/2}b_i)^\ast \phi(a_i^\ast x a_j) (h^{1/2}b_j)
\end{eqnarray*}
is nuclear (as a map into $B$) by Remark \ref{r:nucopconcon}, since $\phi$ is nuclear. 
Since $D$ is a hereditary $C^\ast$-subalgebra of $B$, it follows that the corestriction $d^\ast \pi_\phi(-) d \colon A \to D$ is nuclear for every $d\in D$.

The ``in particular'' part follows from Proposition \ref{p:weaklynucmaps}(2).
\end{proof}

\begin{lemma}
Let $A,B$ be $C^\ast$-algebras, $M$ a von Neumann algebra and let $\phi\colon A \to M$ and $\psi \colon B \to M$ be c.p.~maps with commuting images. Suppose that $A$ is nuclear and $\psi$ is weakly nuclear. Then the induced c.p.~map $\phi \times \psi \colon A \otimes B \to M$ is weakly nuclear.
\end{lemma}
\begin{proof}
Represent $M \subseteq B(H)$. By \cite[Theorem 3.8.5]{BrownOzawa-book-approx}, $\psi \times \iota_{M'}\colon B \otimes_{\max{}} M' \to B(H)$ factors through $B \otimes M'$. The result in \cite{BrownOzawa-book-approx} only considers the case of unital c.p.~maps, but the general case can be obtained by normalising and unitising the map.

When identifying $A \otimes (B \otimes M')$ with $(A \otimes B) \otimes M'$, the c.p.~map $\phi \times (\psi \times \iota_{M'})$ becomes a c.p.~extension of the linear map
\[
(\phi \times \psi) \times_{\mathrm{alg}} \iota_{M'} \colon (A \otimes B) \otimes_{\mathrm{alg}} M' \to B(H).
\]
Thus, by \cite[Theorem 3.8.5]{BrownOzawa-book-approx}, $\phi \times \psi \colon A \otimes B \to M$ is weakly nuclear.
\end{proof}

The following corollary shows that the c.p.~maps that we construct from nuclear maps remain nuclear.

\begin{corollary}\label{c:inducednuc}
Let $A$, $B$ and $D$ be $C^\ast$-algebras, $\phi \colon A \to D^{\ast \ast}$ and $\psi \colon B \to D^{\ast \ast}$ be c.p.~maps with commuting images such that $\phi(A) \psi(B) \subseteq D$. If $A$ is nuclear and $\psi$ is weakly nuclear, then $\phi \times \psi \colon A \otimes B \to D$ is nuclear.
\end{corollary}
\begin{proof}
By the above lemma, there is an induced weakly nuclear map $\phi \times \psi \colon A \otimes B \to D \subseteq D^{\ast \ast}$. Thus the result follows from Proposition \ref{p:weaklynucmaps}(1).
\end{proof}

\begin{remark}\label{r:onetoone}
Recall, by \cite[Corollary 3.1]{WinterZacharias-orderzero}, that there is a one-to-one correspondence between contractive order zero maps $A \to B$ and $\ast$-homomorphisms $C_0(0,1] \otimes A \to B$. 
The bijection is given as follows: if $\rho \colon C_0(0,1] \otimes A \to B$ is a $\ast$-homomorphism, then the induced contractive order zero map, is the composition
\[
A \xrightarrow{f\otimes id_A} C_0(0,1] \otimes A \xrightarrow{\rho} B,
\]
where $f(t)=t$. 

If $\phi \colon A \to B$ is a contractive order zero map, then let $\pi_\phi \colon A \to \multialg{D}$ and $h\in \multialg{D}$ be as in Remark \ref{r:dilationher}. Then $h$ is a positive contraction so there is a canonical $\ast$-homomorphism $\overline \rho \colon C_0(0,1] \to C^\ast(h)\subseteq \multialg{D}$ whose image commutes with the image of $\pi_\phi$. Thus there is a $\ast$-homomorphism $\rho_\phi := \overline \rho \times \pi_\phi \colon C_0(0,1] \otimes A \to \multialg{D}$ and it turns out that this map factors through $D\subseteq B$.
\end{remark}

\begin{theorem}\label{t:nuconetoone}
Let $A$ and $B$ be $C^\ast$-algebras. The one-to-one correspondence between contractive order zero maps $A\to B$ and $\ast$-homomorphism $C_0(0,1] \otimes A \to B$ from \cite{WinterZacharias-orderzero} restricts to a one-to-one correspondence between nuclear contractive order zero maps and nuclear $\ast$-homomorphisms.
\end{theorem}
\begin{proof}
Clearly any nuclear $\ast$-homomorphism $C_0(0,1] \otimes A \to B$ induces a nuclear c.p.c.~order zero map. So we should prove that for any nuclear contractive order zero map $\phi$, the induced $\ast$-homomorphism $\rho_\phi \colon C_0(0,1] \otimes A \to B$ is nuclear. Let $D,h$ and $\pi_\phi \colon A \to \multialg{D} \subseteq D^{\ast \ast}$ be given as in Remark \ref{r:dilationher}. By Proposition \ref{p:orderzerodilation}, $\pi_\phi$ is weakly nuclear. Let $\overline \rho \colon C_0(0,1] \to C^\ast(h) \subseteq D^{\ast \ast}$ be the canonical $\ast$-homomorphism. Then $\rho_\phi = \overline \rho \times \pi_\phi \colon C_0(0,1] \otimes A \to D^{\ast \ast}$ factors through $D$. Since $C_0(0,1]$ is nuclear, Corollary \ref{c:inducednuc} implies $\rho_\phi \colon C_0(0,1] \otimes A \to D \subseteq B$ is nuclear.
\end{proof}

\begin{remark}
It seems tempting to attempt to prove Theorem \ref{t:nuconetoone} by using the following method: if $\phi\colon A \to B$ is a nuclear, contractive order zero map, then the induced $\ast$-homomorphism $\rho_\phi \colon C_0(0,1] \otimes A \to B$ is given by extending the assignment $\rho_\phi(f^n \otimes a) = \phi(a^{1/n})^n$ linearly for $a\in A_+$, $n\in \mathbb N$, where $f(t) = t$, and extending this to $\rho_\phi$ by continuity. Since $\phi$ factors point-norm by c.p.~maps through matrices, it seems that the same should be true for $\rho_\phi$, by the above construction. However, it is not obvious that we get induced maps going from $C_0(0,1]\otimes A$ into matrix algebras nor that these would be completely positive. The same goes for the maps from matrices and into $B$. The author does not know if this approach could work.
\end{remark}


\section{Quasidiagonal traces on exact $C^\ast$-algebras}

Let $\mathcal Q$ be the universal UHF algebra and let $\tau_{\mathcal Q}$ denote its unique tracial state. We let $\ell^\infty(\mathcal Q)= \prod_{\mathbb N} \mathcal Q$ denote the $C^\ast$-algebra of bounded sequences in $\mathcal Q$. We will use the following result, which follows immediately from \cite[Lemma 3.3]{Dadarlat-qdmorphisms}\footnote{One should of course note, that Dadarlat assumes in \cite{Dadarlat-qdmorphisms} that all nuclear maps are contractive, but we may simply normalise our c.p.~maps and use his result.}.

\begin{proposition}\label{p:allnuc}
Let $A$ be an exact $C^\ast$-algebra. Then any c.p.~map $\phi \colon A \to \ell^{\infty}(\mathcal Q)$ is nuclear.
\end{proposition}

We will always refer to tracial states as simply traces.
Recall, (e.g.~\cite{BrownOzawa-book-approx}) that a trace $\tau$ on a $C^\ast$-algebra $A$ is \emph{amenable} if and only if there is a net of contractive c.p.~maps $\phi_\alpha \colon A \to M_{k(\alpha)}$ such that
\begin{itemize}
\item[$(1)$] $\tau(a) = \lim_{\alpha} \mathrm{tr} \circ \phi_\alpha(a)$ for all $a\in A$,
\item[$(2am)$] $\lim_{\alpha}\| \phi_\alpha(ab) - \phi_\alpha(a) \phi_\alpha(b) \|_{2,\mathrm{tr}}=0$ for all $a,b\in A$.
\end{itemize}
We say that a trace $\tau$ is \emph{quasidiagonal} if there is a net as above satisfying $(1)$ and
\begin{itemize}
\item[$(2qd)$] $\lim_{\alpha}\| \phi_\alpha(ab) - \phi_\alpha(a) \phi_\alpha(b) \|=0$ for all $a,b\in A$.
\end{itemize}

We say that a $C^\ast$-algebra $A$ is \emph{quasidiagonal} if there is a net $\phi_\alpha\colon A \to M_{k(\alpha)}$ of contractive c.p.~maps such that
\begin{itemize}
\item[$(1')$] $\| a \| = \lim_{\alpha} \| \phi_\alpha(a)\|$ for all $a\in A$,
\item[$(2')$] $\lim_{\alpha} \| \phi_\alpha(ab) - \phi_{\alpha}(a) \phi_\alpha(b)\| = 0$, for all $a,b \in A$.
\end{itemize}

If $A$ is unital, then all the c.p.~maps in the above definitions may be taken to be unital. In fact, we will very often use the following proposition to restrict to the unital case.

\begin{proposition}[\cite{Brown-invariantmeans}, Proposition 3.5.10]\label{p:unitaltrace}
If $A$ is a non-unital $C^\ast$-algebra and $\tau_A$ is a trace on $A$, then $\tau_A$ extends uniquely to a trace $\tilde \tau_A$ on the unitisation of $A$.

Moreover, $\tau_A$ is amenable (resp.~quasidiagonal) if and only if $\tilde \tau_A$ is amenable (resp.~quasidiagonal).
\end{proposition}

Let $\omega$ denote a free ultrafilter. We define
\[
\mathcal Q_\omega := \ell^\infty(\mathcal Q) / \{ (x_n) \in \ell^\infty(\mathcal Q) : \lim_{n \to \omega} \| x_n \| = 0\}.
\]
Let $q \colon \ell^\infty(\mathcal Q) \to \mathcal Q_\omega$ be the quotient map.
We get an induced trace $\tau_{\mathcal Q_\omega}$ on $\mathcal Q_\omega$ given by
\[
\tau_{\mathcal Q_\omega} (q(a_1,a_2,\dots)) = \lim_{n \to \omega} \tau_{\mathcal Q}(a_n). 
\]
The following proposition generalises the result \cite[Proposition 1.4]{TikuisisWhiteWinter-QDnuc}, and the proof is basically the same.

\begin{proposition}\label{p:qdcharac}
Let $A$ be a separable, unital and exact $C^\ast$-algebra.
\begin{itemize}
\item[$(i)$] Then $A$ is quasidiagonal if and only if there exists a unital, nuclear embedding $A \hookrightarrow \mathcal Q_\omega$.
\item[$(ii)$] For any trace $\tau_A$ on $A$, the following are equivalent:
\begin{itemize}
\item[$(a)$] $\tau_A$ is quasidiagonal,
\item[$(b)$] there exists a unital, nuclear $\ast$-homomorphism $\theta \colon A \to \mathcal Q_\omega$ such that $\tau_{\mathcal Q_\omega} \circ \theta = \tau_A$,
\item[$(c)$] there exists $\gamma \in (0,1]$, such that for every finite finite subset $F \subset A$ and $\epsilon >0$ there exists a nuclear map $\phi \colon A \to \mathcal Q_\omega$ such that
\begin{eqnarray*}
 \|\phi(ab) - \phi(a) \phi(b)\| &<& \epsilon ,  \quad a,b \in F, \textrm{ and} \\
 \tau_{\mathcal Q_\omega} \circ \phi(a) &=& \gamma \tau_A(a), \quad a\in F. 
\end{eqnarray*} 
\end{itemize}
\end{itemize}
In particular, if $A$ has a faithful quasidiagonal trace, then $A$ is quasidiagonal.
\end{proposition}
\begin{proof}
Except for the proof of $(iic) \Rightarrow (iib)$, the proof is identical to that of \cite[Proposition 1.4]{TikuisisWhiteWinter-QDnuc}, where we use Proposition \ref{p:allnuc} to make sure that all c.p.~maps are nuclear.

The proof of $(iic) \Rightarrow (iib)$ is basically the same as in \cite{TikuisisWhiteWinter-QDnuc}, but needs a minor change to make sure that the constructed $\theta$ is nuclear. Let $\phi_n \colon A \to \mathcal Q_\omega$ be a sequence of nuclear maps which are approximately multiplicative and such that $\tau_{\mathcal Q_\omega} \circ \phi_n = \gamma \tau_A$ for all $n$. Since $\| \phi_n \| = \| \phi_n(1_A)\| \approx \| \phi_n(1_A)\|^2$ for large $n$, we may assume that $\| \phi_n\|$ is bounded, say by $2$.
Thus by the Choi--Effros lifting theorem we may lift each $\phi_n$ to a nuclear map $(\phi_n^{(k)})_{k=1}^\infty \colon A \to \ell^\infty(\mathcal Q)$ of norm at most $2$. Thus, by a standard diagonal argument for ultraproducts, we get a c.p.~map $\phi \colon A \to \ell^{\infty}(\mathcal Q)$ such that when composed with the quotient map onto $\mathcal Q_\omega$ we get an embedding $\theta \colon A \to \mathcal Q_\omega$ such that $\tau_{\mathcal Q_\omega} \circ \theta(a) = \gamma \tau_A(a)$ for all $a\in A$. By Proposition \ref{p:allnuc} $\phi$ is nuclear and thus so is $\theta$. Now one proceeds as in \cite{TikuisisWhiteWinter-QDnuc}, where one notes that the corestriction $\theta \colon A \to \theta(1_A) \mathcal Q_\omega \theta(1_A)$ is also nuclear, since $\theta(1_A) \mathcal Q_\omega \theta(1_A)$ is a hereditary $C^\ast$-subalgebra of $\mathcal Q_\omega$.
\end{proof}

We have the following lemma from \cite{TikuisisWhiteWinter-QDnuc}. The author is grateful to the authors for including this lemma.

\begin{lemma}[\cite{TikuisisWhiteWinter-QDnuc}, Lemma 5.7]\label{l:mainlemma}
Let $A$ be a separable, unital, exact $C^\ast$-algebra in the UCT class, and let $\tau_A$ be a faithful trace on $A$. Suppose there are $\ast$-homomorphisms $\Theta \colon C([0,1]) \to \mathcal Q_\omega$, $\grave \Phi \colon C_0([0,1),A) \to \mathcal Q_\omega$ and $\acute \Phi \colon C_0((0,1],A) \to \mathcal Q_\omega$ such that $\Theta$ is unital and $\grave \Phi$ and $\acute \Phi$ are nuclear and compatible with $\Theta$, and
\begin{eqnarray*}
\tau_{\mathcal Q_\omega} \circ \grave \Phi &=& \tau_{\mathrm{Leb}} \otimes \tau_A, \quad \textrm{and} \\
\tau_{\mathcal Q_\omega} \circ \acute \Phi &=& \tau_{\mathrm{Leb}} \otimes \tau_A.
\end{eqnarray*}
Then for any finite subset $F$ of $A$ and $\epsilon >0$, there exists $N\in \mathbb N$ and a nuclear map $\Psi \colon A \to \mathcal Q_\omega \otimes M_{2N}$ such that
\begin{eqnarray*}
\| \Psi(xy) - \Psi(x)\Psi(y) \| &<& \epsilon , \quad x,y\in F, \\
(\tau_{\mathcal Q_\omega} \otimes \tau_{M_{2N}}) (\Psi(x)) &=& \tfrac{1}{2} \tau_A(x), \quad x\in A.
\end{eqnarray*}
\end{lemma}

For the definitions of compatible maps and $\tau_{\mathrm{Leb}}$, we refer the reader to \cite{TikuisisWhiteWinter-QDnuc}, since we will not be using these definitions directly in this paper.

In order to apply the above lemma, we need a few generalisations of results in \cite{TikuisisWhiteWinter-QDnuc}.

\begin{proposition}\label{p:traceorderzero}
Let $A$ be a separable, unital $C^\ast$-algebra and $\tau_A$ be an amenable trace on $A$. Then there exists a contractive order zero map $\Omega \colon A \to \mathcal Q_\omega$ such that
\[
\tau_{\mathcal Q_\omega}(\Omega(a) \Omega(1_A)^{n-1}) = \tau_A(a), \qquad a\in A, n \in \mathbb N.
\]
If $A$ is exact, then we may choose that $\Omega$ is nuclear.
\end{proposition}
\begin{proof}
The construction of $\Omega$ follows the strategy of \cite[Proposition 3.2]{SatoWhiteWinter-nucdim}, with some differences since we do not restrict to the extremal trace case.
As $A$ is separable and unital and $\tau_A$ is amenable, we may find a sequence of unital c.p.~maps $\phi_n \colon A \to M_{k(n)}$ such that $\tau_A(a) = \lim \mathrm{tr} \circ \phi_n(a)$ and $\| \phi_n(ab) - \phi_n(a)\phi_n(b)\|_{2,\mathrm{tr}} \to 0$ for all $a,b\in A$. By embedding $M_{k(n)} \subseteq \mathcal Q$ unitally in a trace preserving way, we may replace all $M_{k(n)}$ with $\mathcal Q$ and $\mathrm{tr}$ with $\tau_{\mathcal Q}$. Thus we get a unital c.p.~map $(\phi_n) \colon A \to \ell^\infty(\mathcal Q)$.
By Proposition \ref{p:allnuc}, if $A$ is exact then $(\phi_n)$ is nuclear. Let $\phi_\omega$ be the induced unital c.p.~map $A \to \mathcal Q_\omega$, which is nuclear if $A$ is exact. The induced unital c.p.~map into $\mathcal Q_\omega/J_{\mathcal Q,\tau_{\mathcal Q}}$ is a $\ast$-homomorphism by how we chose our $\phi_n$. Here $J_{\mathcal Q,\tau_{\mathcal Q}}$ is the left kernel of $\tau_{\mathcal Q_\omega}$, i.e.~$J_{\mathcal Q,\tau_{\mathcal Q}} = \{ x \in \mathcal Q_\omega : \tau_{\mathcal Q_\omega}(x^\ast x) =0\}$. This is a two-sided, closed ideal in $\mathcal Q_\omega$ since $\tau_{\mathcal Q}$ is a trace.

Since $J_{\mathcal Q,\tau_{\mathcal Q}}$ is a $\sigma$-ideal by \cite[Proposition 4.6]{KirchbergRordam-csjiangsu}, we may find a positive contraction $e \in J_{\mathcal Q,\tau_{\mathcal Q}} \cap C^\ast( \phi_\omega(A))'$ such that $ec = c$ for all $c \in J_{\mathcal Q,\tau_{\mathcal Q}} \cap C^\ast(\phi_\omega(A))$.
Let $\Omega = (1_{\mathcal Q_\omega}-e) \phi_\omega(-) (1_{\mathcal Q_\omega}-e)$. If $A$ was exact then $\phi_\omega$ was nuclear and thus $\Omega$ would also be nuclear.

Let $a,b \in A$ be positive such that $ab=0$. Since the composition of $\phi_\omega$ with the quotient map onto $\mathcal Q_\omega /J_{\mathcal Q,\tau_{\mathcal Q}}$ is a $\ast$-homomorphism, $\phi_\omega(a)\phi_\omega(b) \in J_{\mathcal Q, \tau_{\mathcal Q}}$. It follows that
\[
\Omega(a)\Omega(b) = (1-e)\phi_\omega(a)(1-e)^2 \phi_\omega(b) (1-e) = (1-e)^2\phi_\omega(a)\phi_\omega(b)(1-e)^2 = 0,
\]
and thus $\Omega$ has order zero.

From the Cauchy--Schwarz inequality we have
\begin{eqnarray*}
|\tau_{\mathcal Q_\omega}(y(1_{\mathcal Q_\omega} - \Omega(1_A))) | &\leq & \tau_{\mathcal Q_\omega} (y^\ast y)^{1/2} \tau_{\mathcal Q_\omega}((1_{\mathcal Q_\omega} - \Omega(1_A))^2)^{1/2} \\
&=& 0 , \qquad y\in \mathcal Q_\omega,
\end{eqnarray*}
as $1_{\mathcal Q_\omega} - \Omega(1_A)$ is in the kernel of $\tau_{\mathcal Q_\omega}$. Using $y=\Omega(a) \Omega(1_A)^{n-2}$ for $n>1$, we get the result by induction.
\end{proof}

\begin{lemma}\label{l:mapslemma}
Let $A$ be a separable, unital, exact $C^\ast$-algebra and $\tau_A$ be an amenable trace on $A$. Then there exist $\ast$-homomorphisms
\[
\Theta \colon C([0,1]) \to \mathcal Q_\omega, \quad \grave \Phi\colon C_0([0,1),A) \to \mathcal Q_\omega, \quad \acute \Phi \colon C_0((0,1],A) \to \mathcal Q_\omega,
\]
satisfying the conditions of Lemma \ref{l:mainlemma}.
\end{lemma}
\begin{proof}
\cite[Lemma 2.6]{TikuisisWhiteWinter-QDnuc} is exactly this result, but where $A$ is nuclear instead of exact (and with no amenability assumption on $\tau_A$; this is automatic). If we go through the proof of \cite[Lemma 2.6]{TikuisisWhiteWinter-QDnuc} but where we only assume that $A$ is exact instead of nuclear, and $\tau_A$ is amenable, it is still possible to construct $\ast$-homomorphisms $\Theta, \grave \Phi$ and $\acute \Phi$. We only need to check that these maps are nuclear. We use Proposition \ref{p:traceorderzero} to make sure that $\Omega$ in the proof is nuclear. Since $\grave \Phi$ is constructed as the composition of various maps where one of these is $\acute \Phi$, it suffices to show that $\acute \Phi$ is nuclear.

As in the proof \cite[Lemma 2.6]{TikuisisWhiteWinter-QDnuc}, we fix a positive contraction $k \in \mathcal Q$ (satisfying certain properties). The order zero map $A \to \mathcal Q \otimes \mathcal Q_\omega$ given by $a \mapsto k \otimes \Omega(a)$ is clearly nuclear. Compose this map with a unital embedding $\mathcal Q \otimes \mathcal Q_\omega \hookrightarrow \mathcal Q_\omega$ to obtain a contractive, nuclear order zero map $\phi$. As remarked in \cite[Footnote 11]{TikuisisWhiteWinter-QDnuc}, $\acute \Phi$ is the $\ast$-homomorphism associated to $\phi$ as described in Remark \ref{r:onetoone}. By Theorem \ref{t:nuconetoone}, $\acute \Phi$ is nuclear.
\end{proof}

\begin{theorem}\label{t:mainthm}
Any faithful, amenable trace on a separable, exact $C^\ast$-algebra satisfying the UCT is quasidiagonal.

In particular, a separable, exact, unital, simple $C^\ast$-algebra in the UCT class is quasidiagonal if and only if it has an amenable trace.
\end{theorem}
\begin{proof}
By Proposition \ref{p:unitaltrace}, it suffices to prove the case where our $C^\ast$-algebra $A$ is unital.

Let $F \subset A$ be finite and $\epsilon>0$. By Lemmas \ref{l:mapslemma} and \ref{l:mainlemma}, there is a nuclear map $\Psi \colon A \to \mathcal Q_\omega \otimes M_{2N} \cong \mathcal Q_\omega$ such that
\begin{eqnarray*}
\| \Psi(xy) - \Psi(x) \Psi(y)\| &<& \epsilon, \quad x,y \in F, \\
(\tau_{\mathcal Q_\omega} \otimes \tau_{M_{2N}})(\Psi(x)) &=& \tfrac{1}{2} \tau_A(x), \quad x\in A.
\end{eqnarray*}
It follows from Proposition \ref{p:qdcharac} that $\tau_A$ is quasidiagonal.

For the ``in particular'' part, suppose that $A$ is separable, exact, unital and simple. If $A$ has an amenable trace $\tau_A$, then it is faithful, as $A$ is simple, and thus quasidiagonal. By Proposition \ref{p:qdcharac}, $A$ is quasidiagonal. Conversely, if $A$ is quasidiagonal then there is a unital embedding $\iota \colon A \hookrightarrow \mathcal Q_\omega$ by Proposition \ref{p:qdcharac} (which is basically due to Voiculescu in \cite{Voiculescu-qdhtpy}). Then $\tau_{\mathcal Q_\omega} \circ \iota$ is a quasidiagonal (thus amenable) trace on $A$.
\end{proof}

\begin{example}
In \cite[Theorem 6.2.4]{Brown-invariantmeans}, Brown constructs separable, exact, unital, simple, quasidiagonal $C^\ast$-algebras in the UCT class which contain non-amenable traces. In particular, such $C^\ast$-algebras are non-nuclear, since any trace on a nuclear $C^\ast$-algebra is (uniformly) amenable. Since any separable, exact, unital, simple, quasidiagonal $C^\ast$-algebra contains at least one (necessarily faithful) amenable trace, such $C^\ast$-algebras are examples where Theorem \ref{t:mainthm} applies, and for which the results in \cite{TikuisisWhiteWinter-QDnuc} do not apply. 
\end{example}


\section{Traces on quasidiagonal $C^\ast$-algebras}

This section is dedicated to proving the following theorem.

\begin{theorem}\label{t:tracesonqd}
Let $A$ be a separable, exact, quasidiagonal $C^\ast$-algebra in the UCT class. Then any amenable trace on $A$ is quasidiagonal.
\end{theorem}

The similar result for nuclear $C^\ast$-algebras is proved in \cite{TikuisisWhiteWinter-QDnuc}. 
Their result is proved using their main theorem (the nuclear version of our Theorem \ref{t:mainthm}) together with results of N.~Brown in \cite{Brown-invariantmeans}. 

\begin{example}\label{e:conetrace}
Let $A$ be any separable, exact $C^\ast$-algebra which contains an amenable trace $\tau_A$. Then $C_0((0,1],A)$ is a separable, exact $C^\ast$-algebra which is homotopic to zero. Thus it is in the UCT class, and it is quasidiagonal by \cite{Voiculescu-qdhtpy}. If $\mathrm{ev}_1$ denotes the quotient map which is evaluation at $1$, then $\tau_A \circ \mathrm{ev}_1$ is a quasidiagonal trace on $C_0((0,1],A)$.

Note that we did not need to assume that $A$ was in the UCT class. This observation will be used in Proposition \ref{p:nonuct}.
\end{example}

We will be using the same method for proving Theorem \ref{t:tracesonqd} as in \cite{TikuisisWhiteWinter-QDnuc}, however, there is a minor obstacles to this approach.
The results of interest in the paper of Brown, give criteria for when \emph{any} trace is quasidiagonal, and not just when an amenable trace is quasidiagonal.
However, following the same strategy as Brown will give us the result.

We would like to use the following result of Brown. We will actually not use that the $C^\ast$-algebra $A$ in the statement is a \emph{Popa algebra}, but simply use that these are simple and unital (they are also quasidiagonal). Thus we refer the reader to \cite{Brown-invariantmeans} for the definition of Popa algebras.

Recall, that a $C^\ast$-algebra $E$ is called \emph{residually finite dimensional} (RFD) if for every non-zero $x\in E$ there is a finite dimensional representation $\pi$ of $E$ such that $\pi(x) \neq 0$.

\begin{theorem}[\cite{Brown-invariantmeans}, Theorem 5.3.1 and Remark 6.2.3]\label{t:brown}
Let $E$ be a separable, exact, unital, RFD $C^\ast$-algebra in the UCT class.
There exists a separable, exact Popa algebra $A$ in the UCT class with the following property: for every $\epsilon >0$ there is an embedding $\rho_\epsilon \colon E \hookrightarrow A$, such that for every trace $\tau_E$ on $E$ there is a trace $\tau_A$ on $A$ satisfying
\begin{itemize}
\item[$(1)$] $| \tau_A \circ \rho_\epsilon (x) - \tau_E(x)| \leq \epsilon \| x\|$, for all $x\in E$, and
\item[$(2)$] $\pi_{\tau_A}(A)'' \cong R \overline \otimes \pi_{\tau_E}(E)''$, where $\pi_\tau$ is the GNS representation of $\tau$, and $R$ is the hyperfinite $II_1$-factor. 
\end{itemize}
\end{theorem}

We will also need the following characterisation of amenable traces for locally reflexive $C^\ast$-algebras. By a result of Kirchberg, any exact $C^\ast$-algebra is locally reflexive (see e.g.~\cite[Corollary 9.4.1]{BrownOzawa-book-approx}).

\begin{lemma}[\cite{Brown-invariantmeans}, Corollary 4.3.4]\label{l:locrefamenable}
Let $\tau_A$ be a trace on a separable, unital, locally reflexive $C^\ast$-algebra $A$, and let $\pi_{\tau_A}$ denote the GNS representation of $\tau_A$.
Then $\tau_A$ is amenable if and only if $\pi_{\tau_A}(A)''$ is hyperfinite.
\end{lemma}

When given a trace $\tau$ on a $C^\ast$-algebra $A$, such that $\tau$ vanishes on a two-sided, closed ideal $I$ in $A$, then there is an induced trace $\overline \tau$ on $A/I$. If $\tau$ has some desired property (e.g.~it is quasidiagonal or amenable) then we would often want $\overline \tau$ to have the same property.

Recall, that an extension of $C^\ast$-algebras $0 \to B \to E \to A \to 0$ is called \emph{quasidiagonal} if $B$ has an approximate identity of projections, which is quasi-central in $E$.
In particular, if $B$ is a direct sum of unital $C^\ast$-algebras, then any extension by $B$ is quasidiagonal, since $B$ has an approximate identity of projections which is central in the multiplier algebra $\multialg{B}$. This observation will be used in the proof of Theorem \ref{t:tracesonqd}.

The following proposition is an easy consequence of \cite[Propositions 3.5.8, 3.5.10 and Remark 3.5.9]{Brown-invariantmeans}.

\begin{proposition}\label{p:inducedamtrace}
Suppose that $0 \to B \to E \to A \to 0$ is an extension of $C^\ast$-algebras which is locally liftable (e.g.~if $E$ is locally reflexive), let
$\tau_E$ be a trace on $E$ such that $\tau_E(B) = 0$, and let $\tau_A$ be the induced trace on $A$. If $\tau_E$ is amenable, then $\tau_A$ is amenable.

Moreover, if the extension is quasidiagonal and if $\tau_E$ is quasidiagonal, then $\tau_A$ is quasidiagonal.
\end{proposition}

In the following section, Section \ref{s:removeuct}, we will discuss the possibility of removing the quasidiaogonality assumption on the extension in the above proposition, at least in the case when $E$ is exact. It turns out that this is equivalent to removing the UCT assumptions in Theorem \ref{t:mainthm}.

We can now prove the main theorem of this section.

\begin{proof}[Proof of Theorem \ref{t:tracesonqd}]
We will start by proving that any trace on a separable, exact, unital, RFD $C^\ast$-algebra in the UCT class is quasidiagonal, so let $E$ be such a $C^\ast$-algebra and $\tau_E$ be an amenable trace on $E$.

Let $A$ be a separable, exact Popa algebra in the UCT class (which is simple and unital) provided by Theorem \ref{t:brown}. Fix a finite subset $F \subseteq E$ of contractions and $\epsilon >0$ (to prove that $\tau_E$ is quasidiagonal). We may find an embedding $\rho_\epsilon \colon E \hookrightarrow A$ 
and a trace $\tau_A$ on $A$ such that
\[
|\tau_A \circ \rho_\epsilon (x) - \tau_E(x) | < \epsilon \| x\|, \quad x\in E.
\]
and $\pi_{\tau_A}(A)'' \cong R \overline \otimes \pi_{\tau_E}(E)''$.

Since $\tau_E$ is amenable and $E$ is exact, it follows from Lemma \ref{l:locrefamenable} that $\pi_{\tau_E}(E)''$ is hyperfinite. Thus $\pi_{\tau_A}(A)'' \cong R \overline \otimes \pi_{\tau_E}(E)''$ is hyperfinite and again by Lemma \ref{l:locrefamenable}, $\tau_A$ is amenable. Since $A$ is separable, exact, unital, simple and in the UCT class, it follows from Theorem \ref{t:mainthm} that $\tau_A$ is quasidiagonal. 

Thus, we may find a unital c.p.~map $\phi \colon A \to M_n$ such that $\phi$ is multiplicative up to $\epsilon$ on $\rho_\epsilon(F)$ and such that 
\[
|\tau_A(\rho_\epsilon(x)) - \mathrm{tr}((\phi \circ \rho_\epsilon)(x))| < \epsilon, \quad x\in F.
\]
Thus $\phi \circ \rho_\epsilon$ is contractive, it is multiplicative up to $\epsilon$ on $F$ and 
\[
|\tau_E(x) - \mathrm{tr}((\phi \circ \rho_\epsilon)(x))| < 2\epsilon, \quad x\in F.
\]
It follows that $\tau_E$ is quasidiagonal. This proves the result for RFD $C^\ast$-algebras.

Now, let $A$ be as given in the theorem, and let $\tau_A$ be an amenable trace on $A$. Note that if $A$ is non-unital, then its unitisation is again separable, exact, quasidiagonal and in the UCT class. Thus by Proposition \ref{p:unitaltrace} we may assume that $A$ is unital. As $A$ is separable, unital and quasidiagonal, we may pick a sequence of unital c.p.~maps $\phi_n \colon A \to M_{k(n)}$ which is approximately multiplicative and approximately isometric. 

Let $\phi = (\phi_n) \colon A \to \prod_{n=1}^\infty M_{k(n)}$ and let $E = \bigoplus_{n=1}^\infty M_{k(n)} + C^\ast(\phi(A))$. This gives us an extension
\[
0 \to \bigoplus_{n=1}^\infty M_{k(n)} \to E \xrightarrow{q} A \to 0,
\]
for which $\phi \colon A \to E$ is a unital c.p.~split. Since the ideal $\bigoplus_{n=1}^\infty M_{k(n)}$ is a direct sum of unital $C^\ast$-algebras, the extension is quasidiagonal.
Thus $E$ is separable, exact, unital, RFD and in the UCT class, so the trace $\tau_A \circ q$ on $E$ is amenable by Proposition \ref{p:inducedamtrace}, and thus quasidiagonal by what we proved above. It follows again from Proposition \ref{p:inducedamtrace} that $\tau_A$ is quasidiagonal.
\end{proof}


\section{Separability and UCT assumptions}\label{s:removeuct}

It was asked by N.~Brown \cite[Question 6.7(2)]{Brown-invariantmeans} whether every amenable trace is quasidiagonal. Thus it seems desirable to remove the UCT assumptions from our main Theorem \ref{t:mainthm}, which would also allow us to remove the separability assumption, and would solve the question of Brown for exact $C^\ast$-algebras. This would also imply that the Blackadar--Kirchberg problem for simple $C^\ast$-algebras has an affirmative answer. This is what we will address in this final section.

As in Proposition \ref{p:inducedamtrace}, suppose that $0 \to B \to E \to A \to 0$ is an extension of $C^\ast$-algebras, $\tau_E$ is a trace on $E$ such that $\tau_E(B) = 0$, and let $\tau_A$ be the induced trace on $A$. It is often important to know, that if $\tau_E$ has a certain property, then $\tau_A$ has the same property. If $E$ is exact, then the extension is locally liftable, and thus by Proposition \ref{p:inducedamtrace}, if $\tau_E$ is amenable, so is $\tau_A$.

In Proposition \ref{p:inducedamtrace} we need quasidiagonality of the extension to conclude, that when $\tau_E$ is quasidiagonal then $\tau_A$ is quasidiagonal. We could ask if it is possible to omit quasidiagonality of the extension, as in the amenable trace case. It turns out that this question is equivalent to asking whether or not we can remove the UCT assumption (and separability) of our main theorem, Theorem \ref{t:mainthm}. In fact, we get the following.

\begin{proposition}\label{p:nonuct}
The following are equivalent.
\begin{itemize}
\item[$(i)$] Any amenable trace on an exact $C^\ast$-algebra is quasidiagonal,
\item[$(i')$] Any amenable trace on a separable, exact $C^\ast$-algebra is quasidiagonal,
\item[$(ii)$] For any exact $C^\ast$-algebra $A$, any quasidiagonal trace $\tau_A$ on $A$, and any two-sided, closed ideal $J$ in $A$ with $\tau_A(J) = 0$, the induced trace on $A/J$ is quasidiagonal,
\item[$(ii')$] For any separable, exact $C^\ast$-algebra $A$, any quasidiagonal trace $\tau_A$ on $A$, and any two-sided, closed ideal $J$ in $A$ with $\tau_A(J) = 0$, the induced trace on $A/J$ is quasidiagonal.
\end{itemize}
\end{proposition}

Note that $(i')$ and $(ii')$ are simply just $(i)$ and $(ii)$ with separability assumptions.

\begin{proof}
$(i') \Rightarrow (i)$: Let $A$ be an exact $C^\ast$-algebra and $\tau$ be an amenable trace on $A$. By Proposition \ref{p:unitaltrace} we may assume that $A$ is unital. Let $F \subset A$ be a finite set and $\epsilon >0$. Then $A_0 = C^\ast(F,1_A)$ is a separable, unital, exact
$C^\ast$-algebra and $\tau_A$ restricts to an amenable trace $\tau_0$ on $A_0$. By $(i')$ $\tau_0$ is quasidiagonal, so there is $\phi \colon A_0 \to M_n$ such that
$\phi(ab) \approx_\epsilon \phi(a)\phi(b)$ and $\mathrm{tr} \circ \phi(a) = \tau_0(a)$ for all $a,b\in F$. By Arveson's extension theorem, $\phi$ extends to $A$, and thus $\tau_A$ is quasidiagonal on $A$.

$(i) \Rightarrow (ii)$: By Proposition \ref{p:inducedamtrace}, the induced trace on $A/J$ is amenable. Since quotients of exact $C^\ast$-algebras are exact by \cite{Kirchberg-CAR}, it follows from $(i)$ that the induced trace on $A/J$ is quasidiagonal.

$(ii) \Rightarrow (ii')$: Obvious.

$(ii') \Rightarrow (i')$: Let $A$ be a separable, exact $C^\ast$-algebras and $\tau_A$ be an amenable trace on $A$. By Example \ref{e:conetrace} it follows that the trace $\tau_A \circ \mathrm{ev}_1$ on $C_0((0,1],A)$ is quasidiagonal. By $(ii')$ it follows that $\tau_A$ is quasidiagonal.
\end{proof}

\begin{remark}
Note that a similar proof as above shows that we may replace exactness with nuclearity in the above proposition. In this case we may remove the amenability assumption on the traces since this is automatic. Thus we the following.

\emph{Any trace on a (separable) nuclear $C^\ast$-algebra is quasidiagonal if and only if for any (separable) nuclear $C^\ast$-algebra $A$ with quasidiagonal trace $\tau_A$, and any two-sided, closed ideal $J$ such that $\tau_A(J) = 0$, the induced trace on $A/J$ is quasidiagonal.}

The ``separable'' statements above are equivalent to the ``non-separable'' statements.

If one and thus all of these results holds, then the Blackadar--Kirchberg problem has an affirmative answer for simple $C^\ast$-algebras.
\end{remark}

\end{document}